\documentclass[11pt,reqno]{amsart}
\usepackage{amsmath, amssymb, amsthm, wasysym,todonotes, graphicx, manfnt,thmtools,accents}
\usepackage{stmaryrd}
\usepackage{mathrsfs,mathtools}
\usepackage[margin=3cm]{geometry}
\usepackage{color}
\usepackage{xcolor}
\definecolor{cbblue}{RGB}{0,114,178}
\definecolor{cborange}{RGB}{230,159,0}
\usepackage{hyperref}
 \hypersetup{
    colorlinks,
    citecolor=black,
    filecolor=black,
    linkcolor=black,
    urlcolor=black
}

\usepackage[capitalize,nameinlink,noabbrev,nosort]{cleveref}

\usepackage{enumerate}
\usepackage{caption}
\usepackage{float}

\usepackage{tikz}
\usetikzlibrary{shapes.geometric}
\usetikzlibrary{math}
\usetikzlibrary{calc,arrows}
\usepackage{tikz-3dplot}

\newtheorem{theorem}{Theorem}[section]
\newtheorem{lemma}[theorem]{Lemma}

\newtheorem{proposition}[theorem]{Proposition}
\newtheorem{corollary}[theorem]{Corollary}

\theoremstyle{definition}
\newtheorem{definition}[theorem]{Definition}
\newtheorem{example}[theorem]{Example}
\newtheorem{remark}[theorem]{Remark}
\newtheorem*{theorem*}{Theorem}

\newcommand \acknowledgements{\paragraph{\textbf{Acknowledgements}}}

\newcommand{\Z}{\mathbb{Z}}

\newcommand{\R}{\mathbb{R}}

\newcommand{\F}{\mathcal{F}}
\newcommand{\D}{\mathcal{D}}
\newcommand{\precdot}{\prec\cdot ~}
\newcommand{\cube}{\text{\mancube}}

\DeclareMathOperator{\des}{des}

\DeclareMathOperator{\cover}{cover}

\renewcommand{\P}{\mathcal{P}}
\DeclareMathOperator{\cDes}{cDes}

\DeclareMathOperator{\OSP}{OSP}

\newcommand{\dbtilde}[1]{\accentset{\approx}{#1}}

\DeclareMathOperator{\wt}{wt}
\DeclareMathOperator{\Ehr}{Ehr}
\newcommand{\aff}{\mathrm{aff}}

\title{The Ehrhart series of alcoved polytopes}

\author{Elisabeth Bullock}
\address{Massachusetts Institute of Technology}
\email{edb22@mit.edu}
\author{Yuhan Jiang}
\address{Harvard University}
\email{yjiang@math.harvard.edu}

\date{\today}
\begin{document}

\begin{abstract}
Alcoved polytopes are convex polytopes, which are the closure of a union of alcoves in an affine Coxeter arrangement.
They are rational polytopes and, therefore, have Ehrhart quasipolynomials.
Here we describe a method for computing the generating function of the Ehrhart quasipolynomial, or Ehrhart series, of any alcoved polytope via a particular shelling order of its alcoves.
We also show a connection between Early's decorated ordered set partitions and this shelling order for the hypersimplex $\Delta_{2,n}$.
\end{abstract}

\maketitle
\setcounter{tocdepth}{1}

\section{Introduction}

Let $\Phi \subset \R^n$ be an irreducible \emph{crystallographic root system} and $W$ be the corresponding Weyl group. Associated to $\Phi$ is an infinite hyperplane arrangement known as the affine Coxeter arrangement.
The connected components of this hyperplane arrangement are simplices called \emph{alcoves}. 
Lam and Postnikov defined a \emph{proper alcoved polytope} to be the closure of a union of alcoves \cite{alcove2}. 

For any root system, Fomin and Zelevinsky's \emph{generalized associahedra} are examples of polytopes which can be realized as alcoved polytopes \cite{fz03,chapoton}.
In the special situation of the root system $\Phi = A_n$, examples of alcoved polytopes include \emph{hypersimplices} and \emph{positroid polytopes}.

The vertices of alcoved polytopes have rational coordinates, meaning that alcoved polytopes are rational polytopes.
Let $P \subset \R^n$ be a rational convex polytope, with underlying lattice $\Z^n$.
Ehrhart theory is about counting lattice points in rational polytopes.
In particular, Ehrhart showed that the number of lattice points in the $t$-dilate of a rational polytope $P$ is a \emph{quasipolynomial} in $t$ \cite{ehrhart}. 
A quasipolynomial with period $d$ is a function $p: \Z \to \Z$ such that there exist periodic functions $p_i: \Z \to \Z$ with period $d$ so $p(z) = \sum_{i=0}^n p_i(z) z^i$.
We call $E(P,t) = \#(t\cdot P) \cap \Z^n$ the \emph{Ehrhart quasipolynomial} of $P$.
The generating series $\sum_{t=0}^\infty E(P,t) z^t$, called the \emph{Ehrhart series}, is a rational function in $z$.
While the Ehrhart theory of lattice polytopes has been widely studied, the Ehrhart theory of rational polytopes is less explored \cite{beck22,BBKV13,linke11}.

Ehrhart theory extends naturally to polytopes with some facets removed.
The Ehrhart series are \emph{additive}, in the sense that if two half-open polytopes are disjoint, then the Ehrhart series of their union is equal to the sum of their Ehrhart series.

Our main result stated imprecisely is the following:
\begin{theorem*}\label{thm:main-imprecise}
    Fix an irreducible crystallographic root system $\Phi\subset V$, where $\dim(V)=n$.  Let $P$ be an alcoved polytope and let $\Gamma_P=(V,E)$ be the dual graph to the alcove triangulation of $P$.  Pick some $v_0\in V$ and orient the edges of $\Gamma_P$ so that for all $\{u,w\}\in E$, $u\to w$ if and only if $u$ appears before $v$ in the breadth first search algorithm starting at $v_0$.  There exists a weighting of the edges $E$ and parameters $\ell_1,\cdots,\ell_n$ depending only on $\Phi$ such that the Ehrhart series of $P$ is equal to 
    $$\Ehr(P,z) = \frac{\sum_{w \in V} z^{\wt(w)}}{\prod_{i=0}^n (1-z^{\ell_i})}$$
    where $\wt(w) = \sum_{u \to w} \wt((u,w))$ is the sum of the weights of the ingoing edges to $w$.
\end{theorem*}
We will prove our main result using the additivity of Ehrhart series. 
We will decompose each alcoved polytopes into disjoint union of half-open alcoves, and then add up all their Ehrhart series.

\begin{figure}[H]
    \centering
    \begin{tikzpicture}[scale=1.5]
        \coordinate (A) at (0,0);
        \coordinate (B) at (2,0);
        \coordinate (C) at (1,1);
        \draw[thick] (A) -- (B) -- (C) -- cycle;
        \draw node [left, below] at (A) {(0,0)};
        \draw node [right, below] at (B) {(1,0)};
        \draw node [above] at (C) {($\frac{1}{2}, \frac{1}{2}$)};
    \end{tikzpicture}
    \hspace{3cm}
    \begin{tikzpicture}[scale=1.5]
    \coordinate (A) at (0,0);
    \coordinate (B) at (2,0);
    \coordinate (C) at (2,2);
    \coordinate (D) at (3,1);
    \coordinate (E) at (1,1);
    \draw[thick] (A) -- (B) -- (D) -- (C) -- cycle;
    \draw[thick] (E) -- (B);
    \draw[thick] (B) -- (C);
    \coordinate (1) at (1,1/2);
    \coordinate (2) at (3/2,1);
    \coordinate (3) at (5/2,1);
    \draw node at (1) {$\bullet$};
    \draw node at (2) {$\bullet$};
    \draw node at (3) {$\bullet$};
    \draw[thick] (1) -- (2) node [midway, above] {1};
    \draw[thick] (2) -- (3) node [midway, above] {2};
    \end{tikzpicture}
    \caption{On the left, we have the fundamental alcove of type $B_2$. On the right, we have an alcoved polytope of type $B_2$ and the graph of its alcoved triangulation. The Ehrhart series of the square is $\frac{1+z+z^2}{(1-z)^2 (1-z^2)}$.}
    \label{fig:B22}
\end{figure}

\begin{figure}[H]
    \centering
    \begin{tikzpicture}
    \coordinate (A) at (0,0);
    \coordinate (B) at ($(0,{-2*tan(60)})$);
    \coordinate (C) at ($(B) - (2,0)$);

    \draw[thick] (A) -- (B) -- (C) --cycle;
    \draw node [above] at (A) {$(0,0,0)$};
    \draw node [right] at (B) {$(\frac{1}{2},0,-\frac{1}{2})$};
    \draw node [left] at (C) {$(\frac{2}{3}, -\frac{1}{3}, -\frac{1}{3})$};
    \end{tikzpicture}
    \hspace{3cm}
    \begin{tikzpicture}
    \coordinate (A) at (0,0);  
    \coordinate (B) at (2,0);
    \coordinate (C) at ($(A) - (0,{2*tan(60)})$); 

    \coordinate (D) at (-2,0);
    \coordinate (E) at ($(C) - (2,0)$);

    \draw[thick] (A) -- (B) -- (C) -- (E) -- (D) --cycle;

    \draw[thick] (E) -- (A) -- (C);

    \coordinate (F) at ($ (A)!.333!(B)!.333!(C) $);
    \coordinate (G) at ($ (A)!.333!(E)!.333!(C) $);
    \coordinate (H) at ($ (A)!.333!(D)!.333!(E) $);

    \draw[thick] (F) -- (G) node [midway, above] {3};
    \draw[thick] (G) -- (H) node [midway, above] {2};
    \draw node at (F) {$\bullet$};
    \draw node at (G) {$\bullet$};
    \draw node at (H) {$\bullet$};
    \end{tikzpicture}
    \caption{On the left, we have the fundamental alcove of type $G_2$. On the right, we have an alcoved polytope of type $G_2$. The Ehrhart series of the trapezoid is $\frac{1+z^2+z^3}{(1-z)(1-z^2)(1-z^3)}$.}
    \label{fig:G22}
\end{figure}
After proving this result, we show a relationship between this formula and another formula for the $h^*$-polynomial of the second hypersimplex 
$$\Delta_{2,n}=\{(x_1,\dots,x_n)\in[0,1]^n~|~\sum_{i=1}^nx_i=2\}$$ 
which is given in terms of combinatorial objects called decorated ordered set partitions. 
The only known proofs of the latter formula simply enumerate decorated ordered set partitions and show that the number of these objects gives the $h^*$ coefficients; our proof gives a bijective reason for why decorated ordered set partitions appear in the $h^*$-polynomial of the second hypersimplex.

\subsection{Organization}
In \cref{sec:prelim}, we cover relevant background materials, including root systems, alcoved polytopes, and Coxeter groups.
In \cref{sec:ehr}, we review some general Ehrhart theory of rational polytopes that are relevant to alcoved polytopes and prove the lemmas that will be used to show our main result.
Given that Coxeter complexes characterize alcoves, in \cref{sec:shell}, we review Coxeter complexes and the notion of a convex subset of a Coxeter group. In this section, we also describe a shelling order of the Coxeter complex given by \cite{bjorner} and prove a corollary about shelling subcomplexes of the Coxeter complex (\cref{cor:cox}) which will be used to show our main result.
We precisely state and prove our main result (\cref{thm:main}) in \cref{sec:proof}.
In \cref{sec:2n}, we recall a formula for the $h^*$-polynomial of the hypersimplex $\Delta_{k,n}$ in terms of hypersimplicial decorated ordered set partitions \cite{Kimh*} and show a connection with the formula yielded by our main theorem in the case $k=2$.

\vspace{.5cm}
\acknowledgements{We thank Alex Postnikov for suggesting this topic. We thank Lauren Williams for the helpful discussions and comments on the manuscript. We thank Nick Early for helpful discussions.}

\section{Preliminaries}\label{sec:prelim}

In this section, we recall the relevant background for \emph{root systems} which will be used to define \emph{alcoved polytopes}.
We recall notations from \emph{Coxeter groups} which will be used to describe our shellings of the Coxeter complexes related to the alcove triangulation of an alcoved polytope.
We follow the conventions in \cite{humphreys} and \cite{alcove2}.

\subsection{Root systems}\label{sec:rootsys}

Let $V$ be a real Euclidean space of rank $n$ with nondegenerate symmetric inner product $(\cdot,\cdot)$.
Let $\Phi \subset V$ be an irreducible \emph{crystallographic root system} with a choise of basis of \emph{simple roots} $\alpha_1, \dots, \alpha_n$.
Let $\Phi^+ \subset \Phi$ be the corresponding set of \emph{positive roots}.
The \emph{coweight lattice} $\Lambda^\vee$ is the integer lattice defined by $\Lambda^\vee = \Lambda^\vee(\Phi) = \{ \lambda \in V \mid (\lambda, \alpha) \in \Z, \text{ for all } \alpha \in \Phi\}$.
Let $\omega_1, \dots, \omega_n \subset V$ be the basis dual to the basis of simple roots, i.e., $(\omega_i, \alpha_j) = \delta_{ij}$.
The $\omega_i$ are called the \emph{fundamental coweights}.
They generate the coweight lattice $\Lambda^\vee$.

Let $\rho = \omega_1 + \cdots + \omega_n$. The \emph{height} of a root $\alpha$ is the number $(\rho, \alpha)$ of simple roots that add up to $\alpha$.
Since we assumed that $\Phi$ is irreducible, there exists a unique \emph{highest root} $\theta \in \Phi^+$ of maximal possible height.
For convenience we set $\alpha_0 = -\theta$.
Let $a_0 = 1$ and $a_1, \dots, a_n$ be the positivie integers given by $a_i = (\omega_i, \theta)$.

\subsection{Alcoved polytopes}\label{sec:alcove}

Let $\Phi\subset V$ be a crystallographic root system.  The set of affine hyperplanes of the form  $H_{\alpha,k} = \{\lambda \in V \mid (\lambda, \alpha) = k \}$, where $\alpha\in\Phi^+,k\in\mathbb{Z}$,  divides $V$ into \emph{open alcoves}:
\begin{definition}
An (open) \emph{alcove} is the set
$$ A = \{ \lambda \in V \mid m_\alpha < (\lambda, \alpha) < m_\alpha+1, \text{ for } \alpha \in \Phi^+\} $$
where $m_\alpha$ is a collection of integers associated with the alcove $A$.
A \emph{closed alcove} is the closure of an alcove.
\end{definition}

\begin{definition}\label{def:fundamental}
The \emph{fundamental alcove} is the simplex given by
\begin{align*}
    A_\circ &= \{ \lambda \in V \mid 0 < (\lambda, \alpha) < 1, \text{ for } \alpha \in \Phi^+\} \\
    &= \text{Convex Hull of the points } 0, \omega_1/a_1, \dots, \omega_r/a_r.
\end{align*}
\end{definition}


The \emph{affine Weyl group} $W_{\aff}$ associated with the root system $\Phi$ is generated by the reflections $s_{\alpha,k}: V \to V, \alpha \in \Phi, k \in \Z$, with respect to the affine hyperplanes $H_{\alpha,k}$, and acts simply transitively on the collection of all alcoves \cite{humphreys}.  In particular, the closure of each alcove has the same Ehrhart series, and we can use alcoves as the building blocks for a family of polytopes with natural triangulations in which all top-dimensional simplices have equal volume:

\begin{definition}
An \textit{alcoved polytope} is a polytope which is the union of some collection of faces of alcoves.  A \textit{proper alcoved polytope} is a polytope which is a union of alcoves, equivalently a top-dimensional alcoved polytope.
\end{definition}

In other words, every alcoved polytope is of the form
$$ P = \{ \lambda \in V \mid k_\alpha \leq (\lambda, \alpha) \leq K_\alpha, \text{ for } \alpha \in \Phi^+ \}, $$
where $k_\alpha, K_\alpha$ are two collections of integers indexd by the positive roots $\alpha \in \Phi^+$.

\begin{definition}\label{def:graphlabel}
Let $P$ be an alcoved polytope. We associate a graph $\Gamma_P = (V,E)$ with labeled edges to the alcoved triangulation of $P$. 
We will abuse notation and also use $\Gamma_P$ to denote the simplicial complex of the alcove triangulation of $P$.
The vertex set $V$ consists of closed alcoves in $P$, and the edge set $E$ consists of $(A, A')$ if $A$ and $A'$ share a common facet.

Let $(A, A')$ be an edge of $G$, and let $F = A \cap A'$ be the facet it represents. Then $F$ can be transformed to a facet $F_\circ$ of the fundamental alcove $A_\circ$ under the action of the affine Weyl group. Let $\omega_i/a_i$ be the vertex of $A_\circ$ that does not belong to $F_\circ$. Then $(A,A')$ has \emph{weight} $\ell_i$, denoted $\wt((A,A')) = \ell_i$, in which $\ell_i$ is the least common multiple of the denominators in $\omega_i/a_i\in\mathbb{Q}^n$ for $i=1,\dots,n$ and $\ell_0 = 1$.
\end{definition}

\subsection{Coxeter Groups}
For a root system $\Phi$, the affine Weyl group $W_\aff$ is an example of a Coxeter group:
\begin{definition}
Let $S$ be a set. A matrix $m: S \times S \to \{1,2,\dots,\infty\}$ is called a \emph{Coxeter matrix} if it satisfies
\begin{align*}
    m(s,s') &= m(s',s) \\
    m(s,s') = 1 & \iff s = s'.
\end{align*}
A Coxeter matrix $m$ determines a group $W$ with the presentation 
$$ \langle s \in S \mid (ss')^{m(s,s')} = e \text{ for all } m(s,s') \neq \infty \rangle, $$
where $e$ is the identity element.
The pair $(W,S)$ is called a \emph{Coxeter system} and 
the group $W$ is the \emph{Coxeter group}.
\end{definition}
\begin{definition}
Let $(W, S)$ be a Coxeter system. Each element $w \in W$ can be written as a product of generators $w = s_1 s_2 \cdots s_k$ for $s_i \in S$. If $k$ is minimal among all such expressions for $w$, then $k$ is called the \emph{length} of $w$ (written $\ell(w) = k$) and the word $s_1 s_2 \cdots s_k$ is called a \emph{reduced word} for $w$.
\end{definition}

\begin{definition}
Let $(W, S)$ be a Coxeter system, and let $u, v \in W$. Define the \emph{right weak order} as the partial order $\leq$ on $W$ such that $u \leq w$ if and only if $w = u s_1 s_2 \cdots s_k$ for some $s_i \in S$ such that $\ell(u s_1 s_2 \cdots s_i) = \ell(u) + i$ for all $0 \leq i \leq k$.
\end{definition}

\begin{definition}
Let $W$ be a Coxeter group. For $w \in W$, define the \emph{descent set} $\D(w) = \{s \in S \mid w s < w\}$.
For $J \subseteq S$, let $W_J$ be the subgroup of $W$ generated by the set $J$,
and let $W^J = \{ w \in W \mid w s > w \text{ for all } s \in J\}$ be a system of distinct coset representatives modulo $W_J$.
\end{definition}

\section{Ehrhart series of half-open rational simplices}\label{sec:ehr}

A key ingredient of the proof of our main theorem is to write an alcoved polytope $P$ as a disjoint union of half-open simplices, or simplices with some facets removed.  Since the number of lattice points in $P_1\sqcup P_2$ is the sum of lattice points in $P_1$ and $P_2$, it suffices to look at the Ehrhart series of these \emph{half-open} simplices.  If we remove no facets from an alcove, we may use the following theorem to write its Ehrhart series:

\begin{theorem}[Theorem 1.3, \cite{stanley80}]
\label{rat-simp}
Suppose $A$ is a $k$-simplex in $\R^m$ with rational vertices $\beta_0, \dots, \beta_k$. Let $\ell_i$ be the least positive integer $t$ for which $t\beta_i \in \Z^m$.
Consider the $(k+1) \times (m+1)$-matrix whose rows are the vectors $(\ell_i \beta_i, \ell_i)$.
If the greatest common divisor of all $(k+1)\times(k+1)$ minors of the matrix is equal to 1, then the Ehrhart series of the simplex $S$ is equal to
$$ \Ehr(A, z) = \frac{1}{\prod_{i=0}^k (1-z^{\ell_i})}. $$
\end{theorem}

The above theorem serves as a base case in the proof of the following more general formula for half-open simplices:

\begin{lemma}\label{lem:half-opensimplex}
Suppose $A$ is a $k$-simplex in $\R^m$ with rational vertices $\beta_0, \dots, \beta_k$. Let $\ell_i$ be the least positive integer $t$ for which $t\beta_i \in \Z^m$.
Let $F_i$ be the facet of $A$ that does not contain the vertex $\beta_i$.
Let $\mathcal{F} \subseteq \{0, \dots, k\}$ label a subset of facets of $S$.
Then the Ehrhart series of $A$ with facets $\{F_i \mid i \in \mathcal{F}\}$ removed is
$$ \Ehr(A \setminus \bigcup_{i \in \mathcal{F}} F_i, z) = \frac{ \prod_{i \in \mathcal{F}} z^{\ell_i}}{\prod_{i=0}^k (1-z^{\ell_i})}. $$
\end{lemma}

\begin{proof}
Without loss of generality assume $\mathcal{F}=\{0,\ldots,m\}$.  We induct on $m$.  The base case of the induction is $m=0$, where the statement is true by Theorem \ref{rat-simp}.  By inclusion-exclusion and our inductive hypothesis,
\begin{align*}
    \Ehr\left(A\setminus\bigcup_{i \in \mathcal{F}} F_i, z\right) &=\Ehr(A,z)+\sum_{\mathcal{G}\subseteq\mathcal{F}}(-1)^{|\mathcal{G}|}\Ehr\left(\bigcap_{F\in\mathcal{G}}F,z\right)\\
    &=\frac{1}{\prod_{i=0}^k(1-z^{\ell_i})}+\sum_{\mathcal{G}\subseteq\mathcal{F}}\frac{(-1)^{|\mathcal{G}|}}{\prod_{i\in V(\mathcal{G})}(1-z^{\ell_i})}
\end{align*}
where $V(\mathcal{G})$ is the set of indices of vertices $\bigcap_{F\in\mathcal{G}}F$.  We have $V(\mathcal{G})=\{0,\ldots,k\}\setminus\mathcal{G}$. Clearing denominators,
\begin{align*}
     \Ehr\left(A\setminus\bigcup_{i \in \mathcal{F}} F_i, z\right) &=\frac{1}{\prod_{i=0}^k(1-z^{\ell_i})}\left[\sum_{\mathcal{G}\subseteq\mathcal{F}}(-1)^{|\mathcal{G}|}\prod_{i\in\mathcal{G}}(1-z^{\ell_i})\right].
\end{align*}
Using inclusion-exclusion again, the last sum is equal to $\prod_{i\in\mathcal{F}}z^{\ell_i}$.

\end{proof}

\section{A shelling order}\label{sec:shell}

The goal of this section is to build up to a concrete decomposition of any alcoved polytope into half-open alcoves.  This is done via a \emph{shelling order}.  In particular, we relate the alcove structure to a simplicial complex called the coxeter complex and use existing results about the latter to define a shelling order.

We begin by recalling the notion of Coxeter complexes, following \cite{bjornerbrenti}.
\begin{definition}
An \emph{abstract simplicial complex} $\Delta$ on the vertex set $V$ is a collection $\Delta$ of finite subsets of $V$, called \emph{faces}, such that $x \in V$ implies $\{x\} \in \Delta$ and $F \subseteq F' \in \Delta$ implies $F \in \Delta$. 
The dimension of a face is $\dim F = |F| - 1$, and $\dim \Delta = \sup_{F \in \Delta} \dim F$.
A complex is \emph{pure $d$-dimensional} if every face is contained in some $d$-dimensional face.
In this case, we denote the set of $d$-dimensional faces (called facets) by $\mathscr{F}(\Delta)$.
Two facets $A, A'$ are \emph{adjacent} if $\dim(A \cap A') = d-1$.
If $F \in \Delta$, let $\bar{F}$ be the \emph{simplex} $\{E \mid E \subseteq F\}$.
\end{definition}

\begin{definition}
Let $(W,S)$ be a Coxeter system with $|S|<\infty$. 
We write $(s) := S \setminus \{s\}$, for $s \in S$, 
and $V:= \bigcup_{s \in S} W/W_{(s)}$ for the collection of all left cosets of all maximal parabolic subgroups.
The \emph{Coxeter complex} $\Delta(W, S)$ is by definition the pure $(|S|-1)$-dimensional simplicial complex on the vertex set $V$ with facets $C_w := \{w W_{(s)} \mid s \in S\}$, for $w \in W$.
For a face $F \in \Delta(W, S)$, define its \emph{type} $\tau(F) = \{s \in S \mid F \cap W/W_{(s)} \neq \emptyset\}$.
\end{definition}

Let $(W, S)$ be a Coxeter system and $u,v \in W$. A path from $u$ to $v$ is a sequence $u = w_0 \to w_2 \to \cdots \to w_s = v$ such that $w_{i+1} = w_i s$ for some simple reflection $s \in S$. A subset $K \subseteq W$ is called \emph{convex} if for every $u, v \in K$ we have that any shortest path from $u$ to $v$ lies in $K$.  This mirrors the notion of convexity for unions of closed alcoves:

\begin{proposition}[Prop. 3.5, \cite{alcove2}]\label{prop:convex}
Let $P$ be a bounded subset which is a union of closed alcoves. Then $P$ is a convex polytope if and only if one of the following conditions hold:
\begin{enumerate}[(1)]
    \item For any two alcoves $A, B \subset P$, any shortest path from $A= A_0 \to A_1 \to A_2 \to \cdots \to A_s = B$ lies in $P$. Here $A_i$ are alcoves and $A' \to A''$ means that the closure of the two alcoves $A'$ and $A''$ share a facet.
    \item The subset $W_P = \{w \in W_\aff \mid w(A_\circ) \subset P\}$ of the affine Weyl group is a convex subset.
\end{enumerate}
\end{proposition}

Let $(W, S)$ be a Coxeter system and let $\mathscr{F}(\Delta(W, S))$ denote the set of all facets of $\Delta(W, S)$.
By \cite[pp. 40-44]{bourbaki} and \cite[Chap. 2]{tits}, there is a bijection between $W$ and $\mathscr{F}(\Delta(W, S))$ given by $w \mapsto C_w$.
Two facets $C_w$ and $C_{w'}$ are adjacent if and only if $w' = ws$ for some $s \in S$.
By \cite{humphreys}, there is a bijection between $\mathscr{F}(\Delta(W, S))$ and the set of alcoves in the affine Coxeter arrangement.
This bijection maps $C_e$ to the fundamental alcove $A_\circ$ (see \cref{def:fundamental}), and maps $C_w$ to the alcove $A$ such that there is a shortest path $A_\circ \overset{s_1}{\to} A_1 \overset{s_2}{\to} \cdots \overset{s_k}{\to} A$ where $s_1 s_2 \cdots s_k$ is a reduced word for $w$.
The group $W$ acts on $\Delta(W, S)$ by left translation $w: v W_{(s)} \mapsto wv W_{(s)}$, and this action is \emph{type-preserving}, i.e., $\tau(w(F)) = \tau(F)$ for all faces $F \in \Delta(W, S)$.

The key property of Coxeter complexes of interest in this paper is that they have a \emph{shelling order}:
\begin{definition}\label{def:shelling}
Let $\Delta$ be a pure $d$-dimensional complex of at most countable cardinality. A \emph{shelling} of $\Delta$ is a linear order $A_1, A_2, A_3, \dots$ on the set of facets of $\Delta$ such that
$\bar{A_k} \cap \Delta_{k-1}$ is pure $(d-1)$-dimensional for $k = 2,3,\dots$, where
$\Delta_{k-1} = \bar{A}_1 \cup \cdots \cup \bar{A}_{k-1}$.
\end{definition}

Given a shelling, define the \emph{restriction} of a facet $A_k$ by 
$$\mathscr{R}(A_k) = \{x \in A_k \mid A_k - \{x\} \in \Delta_{k-1}\}.$$

\begin{theorem}[Theorem 2.1, \cite{bjorner}]
Let $(W, S)$ be a Coxeter system, $|S|<\infty$. Then any linear extension of the weak ordering of $W$ assigns a shelling order to the facets of $\Delta(W, S)$.
\end{theorem}

In particular, Bj\"orner showed that the restriction of this shelling is 
$$ \mathscr{R}(C_w) = \{w W_{(s)} \mid s \in \D(w)\} .$$

\begin{corollary}\label{cor:cox}
Let $(W, S)$ be a Coxeter system, $|S|<\infty$. Let $\Gamma_P$ be the subcomplex of $\Delta(W, S)$ induced by a convex subset $P$ of the Coxeter group $W$ that contains the identity element $e$. Any linear extension of the weak order is a shelling of $\Gamma_P$.
\end{corollary}

\begin{proof}
Since $P$ contains the identity and is convex, if $w \in P$, then all the shortest paths from $e$ to $w$ are contained in $P$. That is, all the reduced words of $w$ are contained in $P$. 
Let $C_1, C_2, \dots$ denote a shelling of $\Delta(W, S)$, and let $C_{a_1}, C_{a_2}, \dots$ denote the subsequence of $C_1, C_2, \dots$ consisting of facets that are in $\Gamma_P$.
For $k \geq 2$,
let $\Delta_{k-1} = \bar{C}_1 \cup \cdots \cup \bar{C}_{k-1}$ and let $\Delta'_{k-1} = \bar{C}_{a_1} \cup \cdots \cup \bar{C}_{a_{k-1}}$.
Then by definition, $\Delta'_{k-1} \subseteq \Delta_{k-1}$, so $C_{a_k} \cap \Delta'_{k-1} \subseteq C_{a_k} \cap \Delta_{{a_k}-1}$.
Suppose $C_{a_k} = C_w$ for some $w \in W$.
Since all the reduced words of $w$ are contained in $P$, 
for each $s \in \D(w)$, we have $C_w - \{ w W_{(s)} \} \in \Delta'_{k-1}$, so $C_{a_k} \cap \Delta'_{k-1} = C_{a_k} \cap \Delta_{{a_k}-1}$ is pure $(|S|-1)$-dimensional, concluding the proof.
\end{proof}

This corollary allows us to define \emph{breadth-first search order of $\Gamma$}, which is essential to proving our main result.

\begin{figure}
\centering
    \tdplotsetmaincoords{60}{70}
    \begin{tikzpicture}[tdplot_main_coords, scale=5]

    \coordinate (A) at (1, 0, 0);      
    \coordinate (B) at (1, 1, 0);      
    \coordinate (C) at (0.5, 0.5, 0);  
    \coordinate (D) at (1, 1, 1);      
    \coordinate (E) at (0.5, 0.5, 0.5);
    \coordinate (F) at (1.5, 0.5, 0);  
    \coordinate (G) at (1.5, 0.5, 0.5);
    \coordinate (H) at (1, 0.5, 0.5);

    \node[left,below] at (A) {$(1, 0, 0)$};
    \node[right] at (B) {$(1, 1, 0)$};
    \node[left] at (C) {$(\frac{1}{2}, \frac{1}{2}, 0)$};
    \node[right,above] at (D) {$(1, 1, 1)$};
    \node[above] at (E) {$(\frac{1}{2}, \frac{1}{2}, \frac{1}{2})$};
    \node[right,below] at (F) {$(\frac{3}{2}, \frac{1}{2}, 0)$};
    \node[right] at (G) {$(\frac{3}{2}, \frac{1}{2}, \frac{1}{2})$};

    \draw[thick] (A) -- (G) -- (D) -- (E) -- cycle;
    \draw[thick] (B) -- (D);
    \draw[thick] (A) -- (F) -- (B);
    \draw[thick,dotted] (B) -- (C);
    \draw[thick,dotted] (A) -- (C);
    \draw[thick,dotted] (E) -- (C);
    \draw[thick] (G) -- (F);

    \draw[thick] (A) -- (D);
    \draw[thick] (E) -- (G);
    \draw[thick] (G) -- (B);
    \draw[thick,dotted] (H) -- (B);
    \draw[thick,dotted] (A) -- (B);
    \draw[thick,dotted] (E) -- (B);
\end{tikzpicture}
\begin{tikzpicture}
\coordinate (A) at (0,-1);
\coordinate (B) at (1,0);
\coordinate (C) at (1,1);
\coordinate (D) at (3,0);
\coordinate (E) at (3,1);
\coordinate (F) at (4,-1);

\node at (A) {$\bullet$};
\node at (B) {$\bullet$};
\node at (C) {$\bullet$};
\node at (D) {$\bullet$};
\node at (E) {$\bullet$};
\node at (F) {$\bullet$};

\draw[thick] (A) -- (B) node[midway,above] {2};
\draw[thick] (B) -- (C) node[midway,left] {1};
\draw[thick] (D) -- (E) node[midway,right] {1};
\draw[thick] (C) -- (E) node[midway,above] {2};
\draw[thick] (B) -- (D) node[midway,below] {2};
\draw[thick] (D) -- (F) node[midway,above] {2};
\end{tikzpicture}
\caption{The generalized hypersimplex for $\Phi = B_3$ and $k = 2$. The Ehrhart series of $\Delta^{B_3}_2$ is $\Ehr(\Delta^{B_3}_2,z) = \frac{1+z+3z^2+z^3}{(1-z)^2 (1-z^2)^2}$}
\label{fig:B32}
\end{figure}

\section{Proof of the main theorem}\label{sec:proof}

We can now translate the shelling order of subcomplexes of the Coxeter complex in \cref{cor:cox} into a shelling order of the alcoves of an alcoved polytope $P$, and use this shelling order to compute the Ehrhart series of $P$.
\begin{definition}\label{def:partialorder}
Let $\Gamma = (V,E)$ be an undirected graph, and let $v_0 \in V$ be an arbitrary vertex of $\Gamma$. Define the \emph{breadth-first search order of $\Gamma$ with root $v_0$} as the partial order $(\P_{v_0,\Gamma},\prec)$ on $V$ such that for two distinct vertices $u, v \in V$, $u \prec v$ if and only if there is a shortest path from $v_0$ to $v$ passing through $u$.
\end{definition}

The following is a corollary of \cref{cor:cox}.
\begin{corollary}\label{cor:shelling}
Let $P$ be an alcoved polytope and let $\Gamma_P = (V, E)$ be the graph of the alcoved triangulation of $P$. For any $v_0 \in V$, any linear extension of the partial order $(\P_{v_0,\Gamma_P},\prec)$ is a shelling order of the alcove triangulation of $P$.
\end{corollary}

\begin{proof}
Let $W = W_\aff$ be the affine Weyl group that acts on the alcoves in $P$. Let $W_P = \{w \in W \mid w(A_\circ) \subset \}$, which is convex by \cref{prop:convex}.
Let $v_0$ be the alcove $C_w$ for some $w \in W$. Then $w^{-1}(W_P)$ is a convex subset of $W$ that contains the identity, and $w^{-1}(\mathcal{P}_{v_0, \Gamma_P},\prec)$ is the weak order on the facets of $w^{-1}(\Gamma_P)$. 
By \cref{cor:shelling}, any linear extension of $w^{-1}(\mathcal{P}_{v_0, \Gamma_P},\prec)$ is a a shelling order of $w^{-1}(\Gamma_P)$.
Since $W$ acts transitively on the set of all alcoves,
we conclude the proof.
\end{proof}

\begin{lemma}\label{lem:disjointunion}
Let $P$ be an alcoved polytope.
Let $\Gamma_P = (V,E)$ be the graph of the alcoved triangulation of $P$.
Fix some $v_0 \in V$ and let $(\P_{v_0,\Gamma_P}, \prec)$ be the breadth-first search order on $\Gamma_P$ with root $v_0$.
For each alcove $A$ in $P$, let $\mathscr{I}_A=\{\bar{A} \cap \bar{A'} \mid A' \precdot A \}$ be a subset of facets of the closure $A$. 
Then the set of half-open alcoves $A^\circ = \bar{A} \setminus (\cup \mathscr{I}_A)$, are mutually disjoint, and their union is equal to $P$, i.e.,
$$ P = \bigsqcup_{A \in V} A^\circ. $$
\end{lemma}

\begin{proof}
We prove this relationship by induction. Let $A \in V$ be an alcove in $P$. Let $\mathcal{A} = \{A' \in V \mid A' \prec A \}$ be the set of alcoves that comes before $A$ in the topological sort of $\Gamma_P$.
We show that
\begin{enumerate}[(1)]
    \item $A^\circ \cap \bigcup_{A' \in \mathcal{A}} A'^{\circ} = \emptyset$; \label{disjoint}
    \item $A^\circ \cup \bigcup_{A' \in \mathcal{A}} A'^{\circ} = \bar{A} \cup (\bigcup_{A' \in \mathcal{A}} \bar{A'})$.\label{union}
\end{enumerate}

To show (\ref{disjoint}), assume there exists $A' \in \mathcal{A}$ such that $A^\circ \cap A'^{\circ} \neq \emptyset$. 
Then $\bar{A} \cap \bar{A}' \neq \emptyset$ and it must be contained in a facet of $\bar{A}$ in $\bar{A} \cap \bigcup_{A' \in \mathcal{A}} \bar{A}'$ because $\prec$ is a shelling.
This facet is equal to $\bar{A} \cap \bar{A}'$ for $A' \precdot A$, and therefore will be excluded in $A^\circ$.

To show (\ref{union}), we first use the induction hypothesis that $\bigcup_{A' \in \mathcal{A}} A'^{\circ} = \bigcup_{A' \in \mathcal{A}} \bar{A}'$. 
Then, since $\cup \mathscr{I}_A \subset \bigcup_{A' \in \mathcal{A}} \bar{A}'$, we have
$A^\circ \cup (\bigcup_{A' \in \mathcal{A}} \bar{A}' ) = \bar{A} \cup (\bigcup_{A' \in \mathcal{A}} \bar{A}' ) $.
\end{proof}

We now have all the ingredients we need to compute the Ehrhart polynomial of an arbitrary alcoved polytope:
\begin{theorem}\label{thm:main}
Let $P$ be an alcoved polytope and let $\Gamma_P = (V, E)$ be the edge-weighted graph of its alcoved triangulation (see \cref{def:graphlabel} for details). 
Given any $v_0 \in V$, let $\mathcal{P}_{v_0,\Gamma_P}$ be the breadth-first search order of $\Gamma_P$ with root $v_0$ (see \cref{def:partialorder}).
The Ehrhart series of $P$ is equal to 
$$\Ehr(P,z) = \frac{\sum_{v \in V} z^{\wt(v)}}{\prod_{i=0}^n (1-z^{\ell_i})}$$
where $\wt(v) = \sum_{u \precdot v} \wt((u,v))$ is the sum of the weights of the edges between $v$ and the elements it covers.
\end{theorem}

\begin{proof}[Proof of \cref{thm:main}]
By \cref{lem:disjointunion} and \cref{lem:half-opensimplex} and the \emph{additivity} of Ehrhart series, we conclude the proof.
\end{proof}

\begin{example}
Consider the $n$-dimensional hypercube $\cube_n=[0,1]^n$. The graph of alcoved triangulation of a hypercube is the weak Bruhat graph of the symmetric group $\mathcal{S}_n$. Therefore, $h^*(\cube_n,z) = \sum_{w \in \mathcal{S}_n} z^{\des(w)}$, which is the Eulerian polynomial. This is a well known result from \cite{stanley80}.
\end{example}

\section{The hypersimplex $\Delta_{2,n}$}\label{sec:2n}
We now relate our shelling order formula and a combinatorial formula for a well-studied polytope.

The hypersimplex $\Delta_{k,n}$ is the subset of $[0,1]^n\subset\mathbb{R}^n$ consisting of points $\{(x_1,\dots,x_n) \mid x_1 + \cdots + x_n = k\}$. 
Under the linear transformation
$$ y_i = x_1+ \cdots + x_i, $$
the hypersimplex $\Delta_{k,n}$ can be realized as an alcoved polytope defined by $0 \leq y_i - y_{i-1} \leq 1$ and $y_n = k$ for all $i =1,\dots,n$ with the convention $y_0 = 0$. Let $\Gamma_{2,n}$ be the graph of alcoved triangulation of the hypersimplex $\Delta_{2,n}$ (see \cref{def:graphlabel}). In this section we will appeal to the following characterization of the alcove triangulation of $\Delta_{k,n}$:
\begin{theorem}[\cite{alcove1}]
\label{thm:alcove perm description}
    The alcoves of $\Delta_{k,n}$ are in bijection with permutations $w\in \mathcal{S}_n$ modulo cycle shifts (ie. $[w_1,\ldots ,w_n]~[w_n,w_1,\ldots ,w_{n-1}]$) such that if we take the representative of $w$ with $w_n=n$, $w^{-1}$ has $k-1$ descents.  We write $(w)=(w_1,\ldots,w_n)$ to denote the corresponding long cycle in $\mathcal{S}_n$, and $\Delta_{(w)}$ to denote the corresponding alcove in $\Delta_{k,n}$.  Then $\Delta_{(u)}$ and $\Delta_{(w)}$ are adjacent in $\Gamma_{k,n}$ if and only if there exists $i\in[n]$ such that $u_i-u_{i+1}\neq \pm 1 (\text{mod }n)$ and the cycle $(w)$ is obtained from $(u)$ by switching the positions of $u_i$ and $u_{i+1}$.
    
\end{theorem}

The coefficients of the $h^*$-polynomial of the hypersimplex $\Delta_{k,n}$ were proved to be enumerated by \emph{hypersimplicial decorated ordered set partitions}.

\begin{definition}[\cite{early2017conjectures}]
A \emph{decorated ordered set partition} (DOSP) $((S_1)_{s_1},\dots,(S_p)_{s_p})$ of type $(k,n)$ consists of an ordered set partition $(S_1, \dots, S_p)$ of $[n]$ and a $p$-tuple of integers $(s_1, \dots, s_p) \in \Z^{p}$ such that $\sum_{i=1}^p s_i = k$ and $s_i \geq 1$. We regard them up to cyclic rotation, so 
$$((S_1)_{s_1},(S_2)_{s_2},\dots,(S_p)_{s_p})$$ is the same as $$((S_2)_{s_2},\dots,(S_p)_{s_p},(S_1)_{s_1}).$$ A decorated ordered set partition is \emph{hypersimplicial} if $1 \leq s_i \leq |S_i|-1$ for all $i$.
We denote the set of hypersimplicial decorated ordered set partitions of type $(k,n)$ by $\OSP(\Delta_{k,n})$.

We call each $S_i$ a block and place them on a circle in the clockwise fashion then think of $s_i$ as the clockwise distance between adjacent block $S_i$ and $S_{i+1}$. The \emph{winding vector} of a decorated ordered set partition is an $n$-tuple of integers $(l_1,\dots,l_n)$ such that $l_i$ is the distance of the path starting from the block containing $i$ to the block containing $(i+1)$ moving clockwise. If $i$ and $(i+1)$ are in the same block then $l_i = 0$.
If $l_1 + \cdots + l_n = kd$, then we define the \emph{winding number} to be $d$.
\end{definition}

\begin{figure}
    \centering
\begin{tikzpicture}[scale=3,
  every node/.style={fill=white, inner sep=1pt, font=\small},
  arc/.style={thick, ->, >=Stealth}]

  \definecolor{cbblue}{HTML}{0072B2}
  \definecolor{cborange}{HTML}{E69F00}
  \definecolor{cbgreen}{HTML}{009E73}

  \draw[thick] (0,0) circle (1);

\node (P1) at (0:1) {$\{1,5,6\}_{2}$};

\node (P2) at (360/7:1) {$\{2,9\}_{1}$};

\node (P3) at (2*360/7:1) {$\{4,10,12\}_{1}$};

\node (P4) at (5*360/7:1) {$\{3,7,8,11\}_{3}$};

\node at (180/7:1.1) {$1$};
\node at (540/7:1.1) {$1$};
\node at (180:1.1) {$3$};
\node at (-360/7:1.1) {$2$};

\foreach \angle in {30, 80, 180, 315} {
    \coordinate (A) at ({cos(\angle)}, {sin(\angle)});
    \coordinate (B) at ({cos(\angle-5)}, {sin(\angle-5)}); 

    \draw[->, thick] (A) -- (B); }
\end{tikzpicture}
\caption{The winding vector of $((1,5,6)_2,(3,7,8,11)_3,(4,10,12)_1,(2,9)_1)$ is (6,3,3,2,0,2,0,4,6,4,3,2) and the winding number is 35/7=5.
The $i$-th entry of the winding vector is the circular distance between $i$ and $i+1$ in clockwise direction. 
One can walk from 1 to 2 to \dots $n$ back to 1 in clockwise direction, and the winding number is the number of times that one walks around the circle.}
\label{fig:DOSP}
\end{figure}

\begin{theorem}[\cite{Kimh*}]
Let $h^*(\Delta_{k,n},z) = h_0^*(\Delta_{k,n}) + h_1^*(\Delta_{k,n}) z + \cdots + h_{n-1}^*(\Delta_{k,n}) z^{n-1}$ be the $h^*$-polynomial of the hypersimplex $\Delta_{k,n}$.
The number of hypersimplicial decorated ordered set partitions of type $(k,n)$ and winding number $d$ is $h_d^*(\Delta_{k,n})$.
\end{theorem}

\begin{table}[H]
    \centering
    \begin{tabular}{|c|c|}
    \hline
    winding number & $\OSP(\Delta_{2,4})$ \\ \hline
    0 & $((1234)_2)$ \\ \hline
    1 & $((12)_1(34)_1)$ \\ \hline
    1 & $((14)_1(23)_1)$ \\ \hline
    2 & $((13)_1(24)_1)$ \\ \hline
    \end{tabular}
    \caption{The $h^*$-polynomial of the octahedron $\Delta_{2,4}$ is $1+2z+z^2$.}
\end{table}

In \cite{Kimh*} this result is obtained by directly counting the number of hypersimplicial decorated ordered set partitions and comparing the result to the coefficients in $h^*(\Delta_{k,n},z)$.  In this section of our paper, we use our formula in \cref{thm:main} to give a bijective proof of this result for the hypersimplex $\Delta_{2,n}$.

We associate a hypersimplicial decorated ordered set partition of type $(2,n)$ with winding number one to each edge in $\Gamma_{2,n}$.

\begin{definition}\label{def:edgelabels}
Let $A,A' \subseteq \Delta_{2,n}$ be two alcoves such that $A$ and $A'$ share a common facet. Then $A \cap A'$ is of the form $y_j - y_i = 1$ for some $i \not\equiv j \pm 1 \pmod n$.
We associate the hypersimplicial decorated ordered set partition $([i-1,j]_1,[j-1,i]_1)$ to the edge $(A,A')$ in the graph $\Gamma_{2,n}$.  In terms of the permutation characterization of $\Gamma_{2,n}$ given in \cref{thm:alcove perm description}, this edge corresponds to the transposition $ij\to ji$.
\end{definition}



\begin{definition}\label{def:newDOSP} Let $S^C=[n]\setminus S$.  A hypersimlicial decorated ordered set partition of type $(2,n)$ has the form $((S)_1,(S^C)_1)$; we will drop the decorations when they are clear from context. For two hypersimplicial decorated ordered set partitions $(S,S^C), (T,T^C)$ of type $(2,n)$ with nonzero winding numbers, we define
$$\psi((S,S^C), (T,T^C)) = (S \triangle T, S \triangle T^C),$$
where $S \triangle T = (S \setminus T) \cup (T \setminus S)$ is the symmetric difference of the two sets.

For a collection of $d$ adjacent hypersimplicial decorated ordered set partitions $\{(S_i,S_i^C)\}_{i=1}^d$ of type $(2,n)$ and winding number one, we define 
$$\psi((S_i,S_i^C)_{i=1}^d) = (S_1 \triangle \cdots \triangle S_d, [n] \setminus (S_1 \triangle \cdots \triangle S_d)).$$
\end{definition}

\begin{remark}
If we view the $S_i$'s as a binary vectors in $\{0,1\}^n$, then $\triangle$ is the XOR operator.
\end{remark}

\begin{example}
Consider $((123)_1(456)_1)$, $((234)_1(156)_1)$, and $((345)_1(126)_1)$ of type $(2,6)$ and winding number one, we have
\begin{align*}
    &\psi(((123)_1(456)_1),((234)_1(156)_1),((345)_1(126)_1)) \\
    &= \psi(((14)_1(2356)_1), ((345)_1(126)_1)) \\
    &= ((135)_1(246)_1).
\end{align*}
\end{example}

\noindent We can finally state the main result of this section:
\begin{theorem}
\label{thm:hypersimplicial decorated ordered set partition2n}
For any $n$ and for any alcove $A_0$ in $\Delta_{2,n}$, let $\mathcal{P}_{A_0}$ be the breadth-first search order of $\Gamma_{2,n}$ with root $A_0$ (see \cref{def:partialorder}).
For an alcove $A$ in $\Delta_{2,n}$, let $\cover(A)$ be the number of alcoves $A$ covers in the poset $\mathcal{P}_{A_0}$.
In other words, $\cover(A)$ is the number of incoming edges of $A$ in the directed graph we obtain from the Hasse diagram of $\mathcal{P}_{A_0}$.
Applying the map $\psi$ to the set of incoming edges of an alcove gives a bijection from the set of alcoves $A$ with $\cover(A) = d$ and the set of hypersimplicial decorated ordered set partitions of type $(2,n)$ with winding number $d$.
\end{theorem}

\begin{figure}
    \centering
    \begin{tikzpicture}[scale=3.5]
    \node[inner sep=2pt, scale=.6] (4123) at (0:1) {$((35)_1(124)_1)$}; 
    \node[inner sep=2pt, scale=.6] (1243) at (72:1) {$((24)_1(135)_1)$}; 
    \node[inner sep=2pt, scale=.6] (1324) at (144:1) {$((13)_1(245)_1)$};
    \node[inner sep=2pt, scale=.6] (2134) at (216:1) {$((25)_1(134)_1)$};
    \node[inner sep=2pt, scale=.6] (2341) at (288:1) {$((14)_1(235)_1)$}; 
    
    \node[inner sep=2pt, scale=.6] (1423) at (36:0.6) {$((34)_1(125)_1))$}; 
    \node[inner sep=2pt, scale=.6] (3124) at (108:0.6) {$((23)_1(145)_1)$}; 
    \node[inner sep=2pt, scale=.6] (1342) at (180:0.6) {$((12)_1(345)_1)$};
    \node[inner sep=2pt, scale=.6] (2314) at (252:0.6) {$((15)_1(234)_1)$};
    \node[inner sep=2pt, scale=.6] (3412) at (324:0.6) {$((45)_1(123)_1)$}; 

    \node[inner sep=2pt, scale=.6] (3142) at (0,0) {$((12345)_2)$};

    \draw[thick,->] (3142) -- (1423);
    \draw[thick,->] (3142) -- (3124);
    \draw[thick,->] (3142) -- (1342);
    \draw[thick,->] (3142) -- (2314);
    \draw[thick,->] (3142) -- (3412);

    \draw[thick,->,cborange] (3124) -- (1324) node[midway,above=.2,scale=0.6] {$((12)_1(345)_1)$};
    \draw[thick,->] (3124) -- (1243);
    
    \draw[thick,->] (1423) -- (4123);
    \draw[thick,->] (1423) -- (1243);
    
    \draw[thick,->] (3412) -- (4123);
    \draw[thick,->] (3412) -- (2341);
    
    \draw[thick,->] (2314) -- (2134);
    \draw[thick,->] (2314) -- (2341);

    \draw[thick,->] (1342) -- (2134);
    \draw[thick,->,cborange] (1342) -- (1324) node[midway,left,scale=0.6] {$((23)_1(145)_1)$};
\end{tikzpicture}

    \caption{$\Delta_{2,5}$ with $A_0$ at the center.
    The arrows indicate cover relations in $\mathcal{P}_{A_0}$, pointing in increasing directions.
    The orange arrows are facets representing the cover relations of the alcove labeled by $((13)_1(245)_1)$.}
    \label{fig:hypersimplex25}
\end{figure}

\begin{example}
In \cref{fig:hypersimplex25}, we choose $A_0$ to be the simplex in the center of the graph $\Gamma_{2,5}$ and we label it by the unique hypersimplicial decorated ordered set partition of type $(2,5)$ with winding number 0, which is $((12345)_2)$.
The arrows are cover relations, and they point in increasing directions in $\mathcal{P}_{A_0}$. The alcove labeled by $((13)_1(245)_1)$ covers two alcoves in the poset $\mathcal{P}_{A_0}$, through facets colored by orange labeled by $((12)_1(345)_1)$ and $((23)_1(145)_1)$.
\end{example}

We begin by giving a necessary condition for two edges $\{i_1,j_1\}$ and $\{i_2,j_2\}$ to appear together as incoming edges of an alcove $A$ in $\mathcal{P}_{A_0}$.
\begin{definition}
     We say $\{i_1,j_1\}$ and $\{i_2,j_2\}$ are crossing if $i_1<i_2<j_1<j_2$ in the standard cyclic order of $[n]$, and noncrossing otherwise.
\end{definition}
In other words, if we place $1,\dots,n$ in clockwise order on a circle, then the chord $i_1 j_1$ crosses with the chord $i_2 j_2$ in the interior of the circle, as shown in \cref{fig:crossing_chords}.

\begin{lemma}\label{lem:incoming_edge_descriptions}
    Let $A$ be an alcove whose incoming edges include distinct edges $\{i_1,j_1\}$ and $\{i_2,j_2\}$. Then $\{i_1,j_1\}$ and $\{i_2,j_2\}$ are crossing.
\end{lemma}

\begin{proof}
    Write $S_1=[i_1+1,j_1]$ and $S_2=[i_2+1,j_2]$.  If $\{i_1,j_1\}$ and $\{i_2,j_2\}$ are noncrossing, we may assume without loss of generality that $S_1\subsetneq S_2$.  $A$ either has a unique vertex $v$ with $\sum_{i\in S_1}v_i=2$ or a unique vertex $v$ with $\sum_{i\in S_1}v_i=0$.  In the former case, $v$ also satisfies $\sum_{i\in S_2}v_i=2$, and $v$ is the unique vertex of $A$ for which this is true, but this implies that $\{i_1,j_1\}$ and $\{i_2,j_2\}$ give the same edge.  Thus $A$ has a unique vertex $\Tilde{v}$ such that $\sum_{i\in S_1}\Tilde{v}_i=0$ and all other vertices $v$ of $A$ satisfy $\sum_{i\in S_1}v_i=1$.  Similarly, if there is a unique vertex $v$ of $A$ satisfying $\sum_{i\in S_2}v_i=0$, then $v=\Tilde{v}$, again implying $\{i_1,j_1\}$ and $\{i_2,j_2\}$ are the same.  Thus there exists a unique vertex $\dbtilde{v}$ of $A$ such that $\sum_{i\in S_2}\dbtilde{v}_i=2$ and all other vertices $v$ satisfy $\sum_{i\in S_2}v_i=1$. We now know that every vertex $v$ of the root alcove $A_0$ satisfies $\sum_{i\in S_1}v_i\geq 1$ and $\sum_{i\in S_2}v_i\leq 1$. This implies all vertices $v$ of $A_0$ satisfy $\sum_{i\in S_1}v_i=\sum_{i\in S_2}v_i=1$, which implies that $A_0$ is not full dimensional, yielding a contradiction.
\end{proof}

\begin{figure}
    \centering
\begin{tikzpicture}
    \def\r{2} 

    \def\aA{20}
    \def\aC{80}
    \def\aB{140}
    \def\aF{180}
    \def\aD{220}
    \def\aE{320}

    \draw[thick, cbblue]   ({\r*cos(\aA)}, {\r*sin(\aA)}) arc[start angle=\aA, end angle=\aC, radius=\r];
    \draw[thick, cborange] ({\r*cos(\aC)}, {\r*sin(\aC)}) arc[start angle=\aC, end angle=\aB, radius=\r];
    \draw[thick, cbblue]   ({\r*cos(\aB)}, {\r*sin(\aB)}) arc[start angle=\aB, end angle=\aF, radius=\r];
    \draw[thick, cborange] ({\r*cos(\aF)}, {\r*sin(\aF)}) arc[start angle=\aF, end angle=\aD, radius=\r];
    \draw[thick, cbblue]   ({\r*cos(\aD)}, {\r*sin(\aD)}) arc[start angle=\aD, end angle=\aE, radius=\r];
    \draw[thick, cborange] ({\r*cos(\aE)}, {\r*sin(\aE)}) arc[start angle=\aE, end angle=\aA+360, radius=\r];

    \coordinate (A) at ({\r*cos(\aA)}, {\r*sin(\aA)});
    \coordinate (B) at ({\r*cos(\aB)}, {\r*sin(\aB)});
    \coordinate (C) at ({\r*cos(\aC)}, {\r*sin(\aC)});
    \coordinate (D) at ({\r*cos(\aD)}, {\r*sin(\aD)});
    \coordinate (E) at ({\r*cos(\aE)}, {\r*sin(\aE)});
    \coordinate (F) at ({\r*cos(\aF)}, {\r*sin(\aF)});

    \draw[thick] (A) -- (F);
    \draw[thick] (B) -- (E);
    \draw[thick] (C) -- (D);

    \foreach \i in {1,...,18} {
    \node[inner sep=2pt] at ({1.15*\r*cos(-360/18*\i)}, {1.15*\r*sin(-360/18*\i)}) {{\tiny \i}};}

    \foreach \i in {-2,...,0} {
    \node[circle, fill=cborange, inner sep=2pt] at ({\r*cos(360/18*\i)}, {\r*sin(360/18*\i)}) {};}
    \foreach \i in {-7,...,-3} {
     \node[circle, fill=cbblue, inner sep=2pt] at ({\r*cos(360/18*\i)}, {\r*sin(360/18*\i)}) {};}
    \foreach \i in {1,...,3} {
    \node[circle, fill=cbblue, inner sep=2pt] at ({\r*cos(360/18*\i)}, {\r*sin(360/18*\i)}) {};}
    \foreach \i in {4,...,6} {
    \node[circle, fill=cborange, inner sep=2pt] at ({\r*cos(360/18*\i)}, {\r*sin(360/18*\i)}) {};}
    \foreach \i in {7,...,8} {
    \node[circle, fill=cbblue, inner sep=2pt] at ({\r*cos(360/18*\i)}, {\r*sin(360/18*\i)}) {};}
    \foreach \i in {9,...,10} {
    \node[circle, fill=cborange, inner sep=2pt] at ({\r*cos(360/18*\i)}, {\r*sin(360/18*\i)}) {};}

\end{tikzpicture}
    \caption{These chords give a DOSP $\{S,S^C\}$ of winding number 3, where vertices in orange regions belong to $S$ and vertices in blue regions belong to $S^C$.}
    \label{fig:crossing_chords}
\end{figure}

\cref{lem:incoming_edge_descriptions} allows us to show in \cref{cor:num_edges=WN} and \cref{lem:injective_labels_to_DOSPs} that $\psi$ is in fact an injective map from collections of incoming edges for some alcove  to DOSPs of the appropriate winding number.

\begin{corollary}\label{cor:num_edges=WN}
    $\psi$ maps collections of $d$ winding number one DOSPs coming from collections of incoming edges to DOSPs of winding number $d$.
\end{corollary}
\begin{proof}
    Without loss of generality, assume the incoming edges are labelled $\{i_1,j_1\},\ldots,(i_d,j_d)$, listed lexicographically.  When $n$ vertices are placed on a circle and labelled 1 through $n$ in clockwise manner, the chords $\{i_1,j_1\},\ldots,(i_d,j_d)$ partition $[n]$ into $2n$ intervals.  Applying $\psi$ yields the DOSP $(S_1,S^C_1)$ where $S$ is the union of every other interval.  See \cref{fig:crossing_chords}.

\end{proof}

\begin{lemma}\label{lem:injective_labels_to_DOSPs}
    Given a DOSP $(S,S^C)$ of winding number $d$, there is a unique way to write $(S,S^C)=\psi((S_1,S_1^C),\ldots,(S_d,S_d^C))$ where $(S_1,S_1^C),\ldots,(S_d,S_d^C)$ are winding number one DOSPs which are ``crossing'' in the sense of the previous lemma (up to permutation of the $(S_i,S_i^C)$).
\end{lemma}

\begin{proof}
    Since $(S,S^C)$ has winding number $d$, we may write $S=I_1\sqcup \ldots\sqcup I_d$ and $S^C=J_1\sqcup\ldots\sqcup J_d$ where each $I_i$ and each $J_i$ is a contiguous interval in the usual cyclic ordering of $[n]$ and $I_1<J_1<I_2<\ldots<J_d$, where $T<T'$ means $\max(T)<\min(T')$.  Then because of the crossing condition the maximi of $I_1,J_1,I_2,\ldots,J_d$ are $\min (S_1)-1,\min(S_2)-1,\ldots,\min(S_d)-1,\max(S_1),\ldots,\max(S_d)$, up to some permutation of the $S_i$.
\end{proof}

\begin{lemma}\label{lem:injective_edge_labels}
    The map sending an alcove $A$ of $\Delta_{2,n}$ to its set of incoming labels with respect to $\mathcal{P}_{A_0}$ is injective.
\end{lemma}

\noindent Before proving this lemma, we introduce some notation from \cite{parisi2024magic}, and prove one more lemma which allows us to locally determine the orientation of each in $\Gamma_{2,n}$ under $\mathcal{P}_{A_0}$ in terms of the associated permutations.

\begin{definition}

    Let $w\in \mathcal{S}_n$.  We say that $i\in[n]$ is a \emph{left cyclic descent} if $i<n$ and $w^{-1}(i)>w^{-1}(i+1)$ or if $i=n$ and $w^{-1}(1)<w^{-1}(n)$.  We write $\cDes_L(w)$ for the set of left cyclic descents of $w$.
\end{definition}
\begin{definition}
    For $w\in \mathcal{S}_n$, define $w^{(c)}$ to be the rotation of $w$ ending in $c$.
\end{definition}
\begin{definition}[\cite{GZshuffle}]
    Let $u\in \mathcal{S}_m$ and $v\in \mathcal{S}_n$ be permutations on disjoint alphabets.  A \emph{shuffle} of $u$ and $w$ is a permutation in $\mathcal{S}_{m+n}$ such that the entries of $u$ appear in order and the entries of $w$ appear in order.  We say a shuffle $w$ of $u$ and $v$ is \emph{trivial} if $w$ is the concatenation of $u$ and $v$ in either order.
\end{definition}
\begin{definition}
    Let $u\in \mathcal{S}_n$ and $I\subset[n]$.  Define $u|_I$ to be the permutation on alphabet $I$ obtained by restricting $u$ to the entries $I$.
\end{definition}

\noindent Let $i \in [n]$. The \emph{i-order} $<_i$ on the set $[n]$ is the total order
$$
i <_i i+1 <_i \cdots <_i n <_i 1 <_i \cdots <_i i-2 <_i i-1.
$$
The \emph{Gale order} on $\binom{[n]}{k}$ (with respect to $<_i$) is the partial order $\leq_i$ defined as follows: for any two $k$-subsets $S = \{s_1 <_i \cdots <_i s_k \} \subseteq [n]$ and $T = \{t_1 <_i \cdots <_i t_k \} \subseteq [n]$, we have $S \leq_i T$ if and only if $s_j \leq_i t_j$ for all $j \in [k]$ \cite{gale}.
\vspace{-5pt}
\begin{lemma}\label{lemma:monotonicdirection}
Let $A$ be an alcove in $\Delta_{2,n}$.
For any $\{i,j\}$ in the incoming edge set of $A$, such that $w_A^{-1}(i)+1 \equiv w_A^{-1}(j) \pmod n$, then we have $\cDes_L(w_{A_0}^{(j)}) <_j \cDes_L(w_{A}^{(j)})$ and $\cDes_L(w_{A_0}^{(i)}) >_i \cDes_L(w_{A}^{(i)})$ in the Gale order on $\binom{[n]}{2}$ with respect to $<_j$ and $<_i$.
\end{lemma}

\begin{proof}
We proceed by induction on the graphical  distance between $A_0$ and $A$.
Suppose the distance between $A_0$ and $A$ is equal to 1, then $w_{A_0}^{(j)} = [i,~\text{--}~u~\text{--}~,j]$ while $w_A^{(j)} = [~\text{--}~u~\text{--}~,i,j]$, where $u$ is some word in $[n]\setminus \{i,j\}$.
Since $i-j \not\equiv \pm 1 \pmod n$, we have $j <_j i-1 <_j i$ and $i <_i j-1 <_i j$, so $\cDes_L(w_{A_0}^{(j)}) = \{j <_j i-1 \} <_j \cDes_L(w_{A}^{(j)}) = \{j <_j i\}$, and $\cDes_L(w_{A_0}^{(i)}) = \{i <_i j\} >_i \cDes_L(w_{A}^{(i)}) = \{i <_i j-1\}$. For any $k \notin \{i,j\}$, we have $\cDes_L(w_A^{(k)}) = \cDes_L(w_{A_0}^{(k)})$.

Suppose this is shown for all $A'$ of graphical distance $< \ell$ to $A_0$, and suppose $w_A$ can be obtained from $w_A'$ by $ji \to ij$.
Then $\cDes_L(w_{A'}^{(j)}) <_j \cDes_L(w_{A}^{(j)})$ and $\cDes_L(w_{A}^{(i)}) <_i \cDes_L(w_{A'}^{(i)})$.
It suffices to show that $\cDes_L(w_{A_0}^{(j)}) <_j \cDes_L(w_{A'}^{(j)})$ and $\cDes_L(w_{A'}^{(i)}) <_i \cDes_L(w_{A_0}^{(i)})$.
In other words, we want to show that for any shortest path from $A'$ to $A_0$ in $\Gamma_{2,n}$, in the correspoinding permutations of the alcoves, $j$ can only move to the left, and $i$ can only move to the right ($\ast$).
We show this by induction.
The induction hypothesis is trivial at $A'$, because otherwise we will go to $A$, which increases the distance to $A_0$, contradicting our assumption.
Moreover, if $w_{A'} = (\dots,k,j,i,l,\dots)$, then $k \in \{i-1,j-1\}$ and $l \in \{i+1,j+1\}$ since $w_{A'}$ has two cyclic descents.
By induction, this shows that that for any $A''$ in any shortes path from $A'$ to $A_0$, if $w_{A''} = (\dots, j, u, i, \dots)$, there exists $k$ and $l$ such that $u$ is a shuffle of $[i-k,i-1]$ and $[j+1,j+l]$.
Whenever $j$ moves to the left, we increase $k$, and whenever $i$ moves to the right, we increase $l$.
From any alcove $A''$ in any shortest path from $A'$ to $A_0$, if $j$ moves to the right or $i$ moves to the left, then we would go from $A''$ to an alcove that is closer to $A'$ but farther from $A_0$, contradicting our assumption.
Therefore, $(\ast)$ is true for all alcoves in the shortest path from $A'$ to $A_0$, concluding the proof.
\end{proof}

\begin{proof}[Proof of Lemma \ref{lem:injective_edge_labels}]

 Let $(w_A)$ and $(w_{A_0})$ be the long cycles associated with the alcove $A$ and $A_0$ respectively.  We want to show that $(w_A)$ can be uniquely reconstructed from its incoming edge set $E\subseteq\binom{n}{2}$.  For each $\{i,j\}\in E$,  consider $\cDes_L(w_A^{(j)})$ and $\cDes_L(w_{A_0}^{(j)})$.  Both sets are of size 2 and contain $j$.  Let $s\neq j$ be the other element of $\cDes_L(w_{A_0}^{(j)})$.  Either $(w_A)=(\ldots, w_{n-2},i,j,w_1,\ldots)$ or $(w_A)=(\ldots,w_{n-2},j,i,w_1,\ldots)$.  In the first case, $i\in\cDes_L(w_A^{(j)})$ and in the second $i-1\in\cDes_L(w_A^{(j)})$.  Thus by \cref{lemma:monotonicdirection} we have the first case if $s\leq_{j} i-1$ and the second if $s\geq_j i$.
Further, we see from the condition that we have two cyclic descents that if $E=\{\{i_1<j_1\},\ldots,\{i_d<j_d\}\}$ with $i_1<i_2<\dots<i_d$, then $i_1,\ldots,i_d$ appear in increasing order in $(w_A)$ (and we also have $w^{-1}(j_p)\equiv w^{-1}(i_p)\pm 1 \pmod n$).  Further, between the pairs $\{i_p,j_p\}$ and \{$i_{p+1},j_{p+1}\}$ we have some shuffle $u_p$ of the permuations $[i_p+1,\ldots,i_{p+1}-1]$ and $[j_p+1,\ldots,j_{p+1}-1]$.  We also know that $u_p$ appears as a subword in some cyclic rotation of $w_{A_0}$, as otherwise we would have an incoming edge $\{\Tilde{i},\Tilde{j}\}$  with $\Tilde{i}\in I_p:=\{i_p+1,\ldots,i_{p+1}-1\},\Tilde{j}\in J_p:=\{j_p+1,\ldots,j_{p+1}-1\}$.  Thus we know $(u_p)=(w|_{I_p\cup J_p})$, and either there is a unique rotation of $(u_p)$ that is a shuffle of $[i_p+1,\ldots,i_{p+1}-1]$ and $[j_p+1,\ldots,j_{p+1}-1]$, in which case we are done, or $(u_p)=(i_p+1,\ldots,i_{p+1}-1,j_p+1,\ldots,j_{p+1}-1)$, so $u_p$ could be either of the two trivial shuffles (in this case assume $|I_p|,|J_p|\neq 0$, otherwise the trivial shuffles coincide).  To decide which trivial shuffle to take, first assume without loss of generality that $j_{p+1}-1=n$, and consider $\cDes_L{w_{A_0}^{(n)}}=\{a,n\}$.  Either $a\leq i_p$ or $a\geq i_{p+1}$, since the elements of $I_p$ appear in order in $w_{A_0}^{(n)}$.

Consider the first case.  In this case we have $\cDes_L(w_{A_0}^{(j_p+1)})=\{\Tilde{a},j_p+1\}$, where $\Tilde{a}\leq a\leq i_p$ or $\Tilde{a}=n$ since we are potentially rotating the subword $[\Tilde{a}+1,\ldots,a]$ of $w_{A_0}^{(n)}$ to be in front of the entry $\Tilde{a}$.  If $u_p=[i_p+1,\ldots,i_{p+1}-1,j_p+1,\ldots,n]$, then $\cDes_L(w_A^{(j_{p}+1)})=\{i_{p+1}-1,j_{p}+1\}$.  Thus by \cref{lemma:monotonicdirection} we would have an incoming edge $\{i_{p+1}-1,j_{p}+1\}$, since $\Tilde{a}\leq_{j_{p+1}} a \leq_{j_{p+1}} i_p <_{j_{p+1}} i_{p+1}-1$.

In the second case, we have $a\geq i_{p+1}$.  In this case, we have $\cDes_L(w_{A_0}^{(i_{p+1})})=\{i_{p+1},\Tilde{a}\}$ where $a\leq\Tilde{a}<n$ since we are potentially rotating the subword $[\Tilde{a}+1,\ldots, n]$ of $w_{A_0}^{(n)}$ to be in front of the entry $\Tilde{a}$.  If $u_p=[j_{p}+1,\ldots,n,i_p+1,\ldots,i_{p+1}-1]$, then $\cDes_L(w_{A}^{(i_p+1)})=\{i_p+1,n\}$.  Thus by \cref{lemma:monotonicdirection} we would have an incoming edge $\{i_p+1,n\}$ since $n>_{i_p+1}\Tilde{a}\geq_{i_{p}+1}i_{p+1}$.

Thus in both cases we can elimate one of the trivial shuffles as an option.

\end{proof}

\begin{figure}
    \centering
\begin{tikzpicture}[scale=2]
\foreach \i in {1,...,6} {
        \coordinate (P\i) at ({360/6*(3-\i)}:1);
        \node at ({360/6*(3-\i)}:1.2) {\i};
    }
    
    \draw[gray] (0,0) circle (1);
    
    \draw[thick] (P1) -- (P4);
    \draw[thick] (P1) -- (P2);
    \draw[thick] (P2) -- (P4);
    \draw[thick] (P2) -- (P5);
    \draw[thick] (P2) -- (P6);
    \draw[thick] (P1) -- (P3);
\end{tikzpicture}
    \caption{There is a way to read a permutation cycle from the vertices of an alcove in $\Delta_{2,n}$ with the visual help of \emph{thrackles} \cite{dLST,FPthrackle}. For each vertex $e_i+e_j$ of $A$, we draw a chord in the circle with endpoints $i,j$. The permutation cycle represented by this figure is $(1,3,4,2,5,6)$. The facet opposite to the vertex $e_i+e_j$ can be an incoming edge only if the chord $ij$ belongs to the cycle.}
    \label{fig:readingfromthrackle}
\end{figure}

\begin{example}
    Consider $\Delta_{2,15}$.
    Let $(w_{A_0})=(1,2,3,4,5,10,6,7,11,8,12,9,13,14,15)$ and $E=\{\{1,8\},\{3,11\},\{4,13\}\}$.  From the condition on edge orientations, we can deduce that $(w_A)=(1,8,~ \text{---} ~u_1~ \text{---} ~,11,3,~ \text{---} ~u_2~ \text{---} ~,13,4,~ \text{---} ~u_3~ \text{---} ~)$, where $u_1$ is a shuffle of $[2]$ and $[9,10]$, $u_2=[12]$, and $u_3$ is a shuffle of $[5,6,7]$ and $[14,15]$.  We know $(u_1)=(w_{A_0}|_{\{2,9,10\}})=(2,10,9)$, and the only rotation of this which is a shuffle of $[2]$ and $[9,10]$ is $u_1=[9,2,10]$.  We also know that $(u_3)=(w_{A_0}|_{\{5,6,7,14,15\}})=(5,6,7,14,15)$.   We have $\cDes_L(w_{A_0}^{(15)})=\{9,15\}$.  Because $9>8=j_3,$ we deduce $u_3=[5,6,7,14,15]$.  Combining these observations, we have $(w_A)=(1,8,9,2,10,11,3,12,13,4,5,6,7,14,15)$.
\end{example}

\begin{proof}[Proof of Theorem \ref{thm:hypersimplicial decorated ordered set partition2n}]

By Corollary \ref{cor:num_edges=WN} and Lemmas \ref{lem:injective_labels_to_DOSPs} and \ref{lem:injective_edge_labels}, we have an injective map from $\{v \in V \mid \cover(v) = d\}$ and the set of hypersimplicial decorated ordered set partitions of type $(2,n)$ with winding number $d$.  Furthermore, it was shown in \cite{ocneanu2013} that the number of hypersimplicial decorated ordered set partitions of type $(2,n)$ is exactly the number of alcoves of $\Delta_{2,n}$, which is the Eulerian number $A(1,n-1)$.  Thus this map is actually a bijection.

\end{proof}

\bibliographystyle{alpha}
\bibliography{main}

\end{document}